\documentclass[12pt]{amsart}
\usepackage[utf8]{inputenc}
\usepackage{hyperref}
\usepackage[T1]{fontenc}

\usepackage{amsmath}
\usepackage[utf8]{inputenc}
\usepackage{amssymb}
\usepackage{amsthm}
\usepackage{graphicx}
\usepackage{color, soul}
\usepackage{centernot}
\usepackage{verbatim}
\usepackage{multirow}
\usepackage{array}   % for \newcolumntype macro
\newcolumntype{L}{>{$}l<{$}} % math-mode version of "l" column type
\usepackage{setspace}
\usepackage{mathrsfs}
\usepackage{enumitem}
\usepackage{tikz}
\usepackage{bigints}
\usepackage{bm}
\usepackage{caption}
\usepackage{subcaption}
\usepackage{url}
\usepackage{imakeidx}
\makeindex[columns=2, intoc]
\usepackage{tikz-cd}
\makeatletter
\renewcommand{\@biblabel}[1]{[#1]\hfill}
\makeatother

\numberwithin{equation}{section}
\newtheorem{theorem}[equation]{Theorem}

\newtheorem{lemma}[equation]{Lemma}
\newtheorem{proposition}[equation]{Proposition}
\newtheorem{conj}[equation]{Conjecture}

\theoremstyle{definition}
\newtheorem{definition}[equation]{Definition}

\newtheorem{remark}[equation]{Remark}
\newtheorem{example}[equation]{Example}
\newtheorem{algorithm}[equation]{Algorithm}

\usepackage{mathtools}
\mathtoolsset{showonlyrefs}

\newcommand{\eps}{\varepsilon}

\newcommand{\OO}{\mathcal{O}}

\newcommand{\CC}{\mathbb{C}}
\newcommand{\RR}{\mathbb{R}}
\newcommand{\QQ}{\mathbb{Q}}
\newcommand{\ZZ}{\mathbb{Z}}
\newcommand{\NN}{\mathbb{N}}
\newcommand{\FF}{\mathbb{F}}

\DeclareMathOperator{\Tr}{Tr}

\DeclareMathOperator{\Res}{Res}
\DeclareMathOperator{\ord}{ord}

\DeclareMathOperator{\Gal}{Gal}

\DeclareMathOperator{\GL}{GL}

\DeclareMathOperator{\Hom}{Hom}

\DeclareMathOperator{\Frob}{Frob}
\DeclareMathOperator{\Ind}{Ind}
\DeclareMathOperator{\End}{End}
\DeclareMathOperator{\Jac}{Jac}

\DeclareMathOperator{\Fil}{Fil}

\DeclareMathOperator{\dR}{dR}

\DeclareMathOperator{\disc}{disc}

\DeclareMathOperator{\one}{\mathbf{1}}

\DeclareFontFamily{U}{wncy}{}
\DeclareFontShape{U}{wncy}{m}{n}{<->wncyr10}{}
\DeclareSymbolFont{mcy}{U}{wncy}{m}{n}
\DeclareMathSymbol{\Sha}{\mathord}{mcy}{"58}

\newcommand{\pp}{\mathfrak{p}}
\newcommand{\PP}{\mathfrak{P}}

\newcommand{\trianglerightneq}{\mathrel{\ooalign{\raisebox{-0.5ex}{\reflectbox{\rotatebox{90}{$\nshortmid$}}}\cr$\triangleright$\cr}\mkern-3mu}}
\newcommand{\triangleleftneq}{\hspace{0.01cm}\mathrel{\reflectbox{$\trianglerightneq$}}\hspace{0.05cm}}

\newcommand{\midx}{\mid}

\renewcommand{\bar}{\overline}
\renewcommand{\tilde}{\widetilde}

\usepackage{geometry}
\title[Motivic pieces of curves]{Motivic pieces of curves: $L$-functions and periods}

\author{Harry Spencer}
\address{University College, London, WC1H 0AY, UK}
\email{harry.spencer.22@ucl.ac.uk}

\begin{document}

\begin{abstract}
    Given a curve $C$ over a number field $K$ equipped with the action of a finite group $G$ by $K$-automorphisms, one obtains a factorisation of $L(C,s)$ into a product of $L$-functions of `motivic pieces of curves' associated to irreducible $G$-representations. We describe an algorithm for explicitly computing values of these $L$-functions, demonstrating implementations in the cases of certain curves with actions by $C_3$, $C_4$ and $D_{10}$. We explain how this algorithm can be used to factor $L$-functions of curves with endomorphisms of Hecke type. 
    
    Towards applications, we explicitly formulate and numerically verify a version of Deligne's Period Conjecture for hitherto-uninvestigated $L$-functions arising from motivic pieces of superelliptic curves.
\end{abstract}

\maketitle

\tableofcontents

\section{Introduction}\label{sec:intro}
Given a curve $C$ over a number field $K$, one associates to it an $L$-function
    \[ L(C,s) = \prod_{(0)\ne\pp\triangleleftneq\OO_K}L_\pp(C,N(\pp)^{-s})^{-1},\]
where $L_\pp(C,T)=\det(1-\Frob_\pp^{-1} \cdot T \mid (T_\ell\Jac(C)\otimes_{\ZZ_\ell}\bar{\QQ}_\ell)^{*,I_\pp})$ for any $\ell\nmid\pp$ and these determinants are computed in practice by counting points on the special fibre of a regular model of $C$ at each prime of $\OO_K$. \textit{A priori}, this is defined only for $\Re(s)>3/2$, but the Hasse--Weil Conjecture predicts that $L(C,s)$ admits analytic continuation to all of $\CC$ and, famously, the Birch--Swinnerton-Dyer Conjecture (BSD) describes the order of vanishing and leading term of the Taylor expansion at $s=1$ in terms of arithmetic invariants.

If, further, $C$ admits an action by $K$-automorphisms of a finite group $G$, then $T_\ell\Jac(C)\otimes_{\ZZ_\ell}\bar{\QQ}_\ell$ decomposes as a direct sum of irreducible $G$-representations. There is a corresponding factorisation---\textit{a priori} for $\Re(s)>3/2$---on the level of $L$-functions:
\begin{equation}
    L(C,s)= \prod_{\tau \text{ irred.}} L(C^\tau,s)^{\dim \tau},
\end{equation}
where $L(C^\tau,s)$ is the $L$-function associated to a `motivic piece', as described in \cite[\S2]{DGKM2} (cf.\ Definition \ref{defn:motivic_piece}). These will be the central objects of our study.

The main achievement of the present work is Algorithm \ref{algo}, an algorithm for computing the Dirichlet series coefficients of $L(C^\tau,s)$. The primary novelty of this is the reduction of the computation of Euler factors to point counts on curves---analogously to the standard algorithm for $L(C,s)$. We implement this algorithm in several cases using the Magma \cite{Magma} adaptation of T. Dokchitser's $L$-function calculator \cite{DokCalc}. There were several motivations for the development of this algorithm, including: explicitly recovering cases where this decomposition matches with known factorisations of $L$-functions, improving $L$-function computations efficiency for curves with extra automorphisms (or endomorphisms), and numerical testing of many conjectures.

On this latter point: formally, the motivic piece $C^\tau$ is a motive over $K$ with coefficients in $\QQ(\tau)$; it is the motive associated to the first cohomology of $C$ and the $\tau$-idempotent of $\CC[G]$. For the purposes of this work, we eschew any categorical discussion of motives---instead adopting the viewpoint of Deligne (cf.\ \cite[(0.12)]{Deligne79}), from which we see motives as a collection of cohomologies $H_B, H_{\dR}, (H_{\ell})_\ell$ with additional data, including comparison isomorphisms between them (see \S\ref{sec:motives}). 

The arithmetic of $L$-functions arising from motives is widely believed to be governed by a vast conjectural edifice which generalises BSD---for example, we are in the precise setting of the equivariant Tamagawa Number Conjecture of Burns and Flach \cite[\S4]{BurnsFlach}, after Bloch and Kato, for the pair $(h^1(C),\CC[G])$. We do not pursue verification of this conjecture herein, instead working on the lower levels of this edifice. At this level, several standard conjectures are interpreted explicitly for the $L$-functions $L(C^\tau,s)$ in \cite{DGKM2} (see [\textit{loc.\ cit.}, Conjectures 2.20 \& 2.23, Theorem 2.24]). See \S\ref{sec:motivic_pieces} for explanations of the terms in \eqref{eqn:func}.

\begin{conj}\label{conjs}
    The $L$-functions $L(C^\tau,s)$ satisfy the following properties:
    \begin{itemize}
        \item[\emph{\textbf{Hasse--Weil:}}] $L(C^\tau,s)$ admits meromorphic continuation to the complex plane and the completed $L$-function $\Lambda(C^\tau,s)$ satisfies the functional equation
        \begin{equation}\label{eqn:func} \Lambda(C^\tau,s)=w(C^\tau)\cdot \Lambda(C^{\tau^*},2-s).\end{equation}
        \item[\emph{\textbf{Deligne:}}] For $\sigma\in\Gal(\QQ(\tau)/\QQ)$, we have $\ord_{s=1} L(C^\tau,s)=\ord_{s=1} L(C^{\sigma(\tau)},s)$.
        \item[\emph{\textbf{Weak BSD:}}] $\ord_{s=1}L(C^\tau,s)=\langle\tau,\Jac(C)(K)\otimes_\ZZ\CC\rangle$.
        \item[\emph{\textbf{Parity:}}] For $\tau$ self-dual, $w(C^\tau)=(-1)^{\langle\tau,\Jac(C)(K)\otimes_\ZZ\CC\rangle}$.
    \end{itemize}
\end{conj}

Further to the conjecture of his appearing above, Deligne makes a conjecture (\cite[Conjecture 2.8]{Deligne79}, cf.\ Conjecture \ref{conj:Deligne}) linking the `irrational parts' of motivic $L$-values to Deligne periods (Definition \ref{defn:del_period}). We refer to this as Deligne's Period Conjecture, numerical verification in new settings of which will be our primary focus. 

To do so, we work with superelliptic curves $y^n=f(x)$, a special case affording an explicit, testable formulation (Conjecture \ref{conj:super_deligne}). In particular, in \S\ref{sec:super_periods} we discuss how the results of \cite{SuperellipticPeriods} and the associated Magma implementation \cite{SuperellipticPeriodsCode} can be used to numerically compute the Deligne periods of motivic pieces of superelliptic curves. In \S\ref{sec:examples}, we combine this implementation with Algorithm \ref{algo} to numerically verify Deligne's Period Conjecture for the motivic pieces of several superelliptic curves.

\begin{theorem}\label{thm:main}
    For the $14$ genus $3$ superelliptic curves with $L(C,1)\ne0$ considered in \S\ref{sec:examples}, there exist $a,b,c\in\ZZ$ with $|a|,|b|\le 16$ and $|c|\le160$ such that
    \begin{equation}\label{eqn:approx}
        \left|\frac{L^{\text{approx}}(C^\tau,1)}{\Omega^\text{approx}_{C^\tau}}-\frac{a+b\cdot\omega}{c}\right| < \frac{1}{c^4},
    \end{equation}
    where $\tau$ is a non-$\QQ$-rational character of $C_3$ or $C_4$, $L^{\text{approx}}(C^\tau,1)$ is our numerical value for $L(C^\tau,1)$, $\Omega^\text{approx}_{C^\tau}$ is that for the Deligne period as in \S\ref{sec:super_periods}---both computed to precision at least $10$---and $\omega=i$ or $\omega=\zeta_3$ depending on the exponent of $C$.
\end{theorem}

\begin{remark}
    The statement of Theorem \ref{thm:main} references approximate values because, although our computations of Dirichlet series coefficients are provably correct, the algorithm of \cite{DokCalc} is not provably accurate without knowing the Hasse--Weil Conjecture for $L(C^\tau,s)$. In all of the examples herein, the functional equation is at least numerically verified to the precision of our calculations.
\end{remark}

\begin{remark}
    Given $\eps>0$, for almost all $\alpha\in\RR$, including all algebraic $\alpha$, there are only finitely many coprime pairs $(p,q)$ for which $|\alpha-p/q|<1/q^{2+\eps}$ (cf.\ \cite[Theorem 32]{KhinchinContFractions}, \cite{Roth}). With this in mind, looking at the real part and the imaginary part (divided by $\sqrt{3}$ if $\omega=\zeta_3$) of \eqref{eqn:approx} separately, we find vastly better approximations than one would expect for a non-$\QQ(\omega)$-rational. Theorem \ref{thm:main} therefore provides compelling evidence for the rationality of $L(C^\tau,1)/\Omega_{C^\tau}$ and for the validity of Deligne's Period Conjecture.
\end{remark}

\begin{example}[{= Example \ref{ex:pic1}}]
    Consider the curve
    \begin{equation}
        C/\QQ(\zeta_3): \hspace{0.2cm} y^3 = x^4+x^3+\zeta_3x^2+\zeta_3-2
    \end{equation}
    and write $\tau$ for a non-trivial character of the automorphism group $C_3\cong\langle y\mapsto\zeta_3\cdot y\rangle$. With $100,000$ Dirichlet series coefficients, we compute $L(C^\tau,1)\approx 1.8460900297\hdots - i\cdot 1.5447030118\hdots$ and $\Omega_{C^\tau}= -38.6473408677\hdots - i\cdot105.727015294\hdots .$
    Together, these computations yield
    \begin{equation}
            \frac{L(C^\tau,1)}{\Omega_{C^\tau}}\approx -0.0185185185185\hdots + i\cdot0.0106916716517\hdots \approx \frac{\zeta_3-1}{81}.
    \end{equation}
\end{example}

\begin{remark}
    Naturally, it would be of interest in future work to go beyond the rationality predicted by Deligne's Period Conjecture to say something about the rational $L$-values it predicts to exist, \textit{{\`a} la} Birch and Swinnerton-Dyer. Between \cite{DokchitserEvansWiersema} and \cite{BurnsMC}, it is explained that $L$-values of Artin-twists of elliptic curves ought to depend on the Galois-module structure of the Tate--Shafarevich group. Utilising this philosophy, the recent work \cite{MaistretShukla2025} in the same setting seeks to explicitly describe the ideal generated by the $L$-value of an Artin-twist in terms of this Galois-module structure. As this setting is a special case of motivic pieces of curves, one ought to believe that $L$-values of motivic pieces of curves depend (at least) on BSD-invariants and associated $G$-module structures.

    It is also worth noting that, in the case of Artin-twisted $L$-functions of abelian varieties, further integral refinements of Deligne's Period Conjecture are an active area of study (e.g.\ \cite{EvansThesis,EvansMCW}) and these may extend to our setting.
\end{remark}

As another application, we demonstrate how the algorithm discussed in \S\ref{sec:algo} can be used to factorise $L$-functions of curves with extra endomorphisms of Hecke type (Definition \ref{def:HeckeType}), rather than extra automorphisms. For example, by studying an action of the dihedral group of order $10$ on a genus $6$ cover, we recover the known factorisations of $L$-functions of a family of genus $2$ curves over $\QQ$ whose Jacobians are of $\GL_2$-type into products of $L$-functions of classical newforms:

\begin{proposition}[{= Proposition \ref{prop:mod_form}}]\label{prop:mod_form_intro}
    Consider the one-parameter family of genus $2$ curves
    \[C_t/\QQ: \hspace{0.2cm} y^2 = x^6-2x^5+x^4+(6-4t)x^3+2x^2+4x+1.\] 
    There is a degree $2$ cover $X_t$ of $C_t$ admitting an action by the dihedral group of order $10$ such that $C_t\cong X_t/C_2$ and so $L(C_t,s)=L(X_t^\rho,s)\cdot L(X_t^{\rho'},s)$, where $\tau$ and $\tau'$ are conjugate irreducible two-dimensional representations. The factors $L(X^\rho_t,s)$ and $L(X^{\rho'}_t,s)$ are the $L$-functions of conjugate newforms with coefficients in $\QQ(\sqrt{5})$.
\end{proposition}

In general, to factorise the $L$-function of a Jacobian $\Jac(C)$ of $\GL_2$-type into $L$-functions of newforms, one must compute the conductor of $C$ and thereby obtain the level of the associated newforms, reducing to a finite brute-force check. Proposition \ref{prop:mod_form_intro} demonstrates how---at least in principle---Algorithm \ref{algo} allows, when $C$ has endomorphisms of Hecke type, direct computation of Fourier coefficients without needing even full knowledge of the conductor of $C$.

\subsection*{Conventions} Unless otherwise stated, $C/K$ denotes throughout a curve over a number field equipped with an action of the finite group $G$ by $K$-automorphisms, with $\tau$ a representation of $G$. All of our computations are carried out in Magma and all code used is available at

\begin{center}
    \url{https://github.com/hspen99/MotivicPiecesOfCurves} \cite{MotivicPiecesCode}.
\end{center}

\subsection*{Acknowledgements} 
This work finds some of its roots in a first-year PhD `mini-project' on which I collaborated with Giorgio Navone and Corijn Rudrum. My utmost thanks goes to the both of them for all of their thoughts and contributions in those early days. Likewise, I thank Vladimir Dokchitser for suggesting that project and for his dedicated supervision, both then and now. I thank Tim Dokchitser for providing the family of genus $2$ curves discussed in \S\ref{sec:dihedral}. I thank Dominik Bullach for a tremendously helpful discussion regarding Deligne's conjecture and for comments on a draft. I also thank John Voight for helpful correspondence.

This work was supported by the Engineering and Physical Sciences Research Council [EP/S021590/1], the EPSRC Centre for Doctoral Training in Geometry and Number Theory (The London School of Geometry and Number Theory), University College, London.

\section{Motivic pieces of curves}\label{sec:motivic_pieces}
The framework of motivic pieces of curves (or abelian varieties) was laid out by V. Dokchitser, Green, Konstantinou and Morgan in their work on the Parity Conjecture for Jacobians of curves \cite[\S2]{DGKM2}. The set-up is as follows: fix a curve $C$ over a number field $K$ with the action of a finite group $G$ by $K$-automorphisms, as well as embeddings $\bar{\QQ}\subset \bar{\QQ}_\ell\subset \CC$ for each prime $\ell$ such that $\bar{\QQ}\subset \CC$ is independent of $\ell$. The $G$-action on $C$ extends to $H^1_\ell(C)=(T_\ell \Jac(C)\otimes_{\ZZ_\ell}\bar{\QQ}_\ell)^*$ for each prime $\ell$.

\begin{definition}[{= \cite[Definition 2.3]{DGKM2}}]\label{defn:motivic_piece}
    For a $G$-representation $\tau$, define
    \[ H^1_\ell(C^\tau)=\Hom_G(\tau,H^1_\ell(C)),\]
    on which $G_K$ acts by postcomposition. Implicitly this definition depends on $\ell$, but in fact it yields a compatible system of $\ell$-adic representations (cf.\ \cite[Corollary 2.14]{DGKM2}).
\end{definition}

This system of $\ell$-adic representations gives rise to an $L$-function in the usual way.

\begin{definition}\label{defn:l_func}
    We write $L(C^\tau,s)$ for the $L$-function $L((H^1_\ell(C^\tau))_{\ell},s)$, given for $\Re(s)>3/2$ by
    \[ L(C^\tau,s) = \prod_{(0)\ne\pp\triangleleftneq\OO_K} L_\pp(C^\tau,N(\pp)^{-s})^{-1}, \]
    where $L_\pp(C^\tau,T):=\det(1-\Frob_\pp^{-1} \cdot T \mid H^1_\ell(C^\tau)^{I_\pp})$ for any choice of $\pp\nmid\ell$. Here $I_\pp$ is the inertia group of the completion of $K$ at $\pp$, and $\Frob_\pp$ is (a choice of) arithmetic Frobenius.
\end{definition}

\begin{remark}\label{rmk:l_func_decomp}
    A less explicit way of constructing the $L$-function $L(C^\tau,s)$ as in Definition \ref{defn:l_func} is as follows: construct the abelian variety $A_{\hat\tau}$ whose $\ell$-adic cohomology is $\oplus_{\sigma\in\Gal(\QQ(\tau)/\QQ)}H^1_\ell(C^{\sigma(\tau)})$---this is a quotient of $\Jac(C)$ appearing with multiplicity $\dim\tau$ in the isogeny decomposition thereof. We have $\QQ(\tau)\hookrightarrow \End(A_{\hat\tau})\otimes\QQ$ and the corresponding decomposition on the level of $L$-functions coming from the fact that $V_\ell(A)$ is a $\QQ(\tau)\otimes\QQ_\ell$-module is precisely that coming from that of $L(C,s)$ into $L$-functions of motivic pieces. This point of view will prove helpful in \S\ref{sec:dihedral}.
\end{remark}

We note that the $\tau$-isotypic piece of $H^1_\ell(C)$ differs from the motivic piece corresponding to $\tau$ when $\dim\tau>1$. This is because, for $\tau$ irreducible, we have a canonical isomorphism
\begin{equation}\label{eqn:isotypic}
        \Hom_G(\tau,V)^{\oplus\dim\tau}\cong V^\tau,
\end{equation}
where $V$ is any $G$-representation and $V^\tau$ is the $\tau$-isotypic piece thereof. This brings us to an important analogue of the `Artin-formalism'.

\begin{lemma}[{= \cite[Proposition 2.8]{DGKM2}}]
    For representations $\tau_1$ and $\tau_2$ of $G$, and $\rho$ a representation of $H\le G$:
    \begin{enumerate}
        \item $H^1_\ell(C^{\tau_1+\tau_2})\cong H^1_\ell(C^{\tau_1})\oplus H^1_\ell(C^{\tau_2})$
        \item $H^1_\ell(C^{\Ind_H^G \rho})\cong H^1_\ell(C^\rho)$.
    \end{enumerate}
\end{lemma}

\begin{proof}
    (1) follows from additivity of $\Hom$. (2) follows from Frobenius reciprocity.
\end{proof}

The following is an important formula for the trace of an element of the absolute Galois group acting on the $\ell$-adic representation of a motivic piece in terms of the Galois action on $H^1_{\ell}(C)$ as a whole.

\begin{proposition}[{= \cite[Lemma 2.12]{DGKM2}}]\label{prop:traces}
    For $g\in G$ and $\alpha\in G_K$, we have
    \begin{equation}\label{eqn:trace}
        \Tr(\alpha\mid H^1_\ell(C^\tau))=\frac{1}{|G|}\sum_{g\in G}\Tr(g^{-1}\mid \tau)\Tr(\alpha\cdot g | H^1_\ell(C)).
    \end{equation}
\end{proposition}

Immediately, this yields:

\begin{lemma}\label{lem:conj}
    For $\sigma\in\Gal(\QQ(\tau)/\QQ)$, we have 
    \[\det(1-\Frob_\pp^{-1}T | H^1_\ell(C^\tau))^\sigma = \det(1-\Frob_\pp^{-1}T | H^1_\ell(C^{\sigma(\tau)})).\]
\end{lemma}

Another consequence of Proposition \ref{prop:traces} is the following fact about conductors, which will prove helpful in our calculations. The conductor of an $\ell$-adic representation is an ideal which---in some sense---measures bad reduction, and happens to be the fundamental invariant associated to $L$-functions. See, for example, \cite{UlmerConductors} for the relevant background.

\begin{lemma}[{= \cite[Lemma 3.10]{WildConductors}}]\label{lem:conductors}
    Let $\tau$ and $\tau'$ be conjugate representations of $G$. Then $C^\tau$ and $C^{\tau'}$ have equal conductor norms.
\end{lemma}

\begin{remark}\label{rmk:GammaCond}
    For the most part, we will be interested in the conductor $N(C^\tau)$ of the restriction of scalars $\Res_{K/\QQ}C^\tau$, rather than the usual conductor $\mathcal{N}(C^\tau)$ of $C^\tau$. This is because Magma works with Dirichlet series over $\QQ$. Note that the two quantities are related by $N=N_{K/\QQ}(\mathcal{N})\cdot|\Delta_{K/\QQ}|^{\langle\tau,H^1_\ell(C)\rangle}$, where $\Delta_{K/\QQ}$ is the discriminant of this extension (cf.\ \cite[p.\ 178]{MilneAV}).
\end{remark}

We write $\Lambda(C^\tau,s)$ for the completed $L$-function
    \[ \Lambda(C^\tau,s) := L_\infty(C^\tau,s)\cdot L(C^\tau,s), \]
where 
\begin{equation}\label{eqn:GammaFactor}
    L_\infty(C^\tau,s)=\left(\frac{N(C^{\tau})}{\pi^{\langle\tau,H^1_\ell(C)\rangle\cdot[K:\QQ]}}\right)^{\frac{s}{2}}\cdot\left(\Gamma\left(\frac{s}{2}\right)\Gamma\left(\frac{s+1}{2}\right)\right)^\frac{\langle\tau,H^1_\ell(C)\rangle\cdot[K:\QQ]}{2},
\end{equation}
which one should think of as the contribution to the Euler product from infinite primes. The Hasse--Weil Conjecture for motivic pieces then reads as follows:

\begin{conj}[Hasse--Weil]\label{conj:HasseWeil}
    $L(C^\tau,s)$ admits meromorphic continuation to $\CC$, and the completed $L$-function $\Lambda(C^\tau,s)$ satisfies the functional equation
    \begin{equation}\label{eqn:func2}
        \Lambda(C^\tau,s)=w(C^\tau)\cdot \Lambda(C^{\tau^*},2-s),
    \end{equation}
    where $w(C^\tau)$ is the root number of $C^\tau$.
\end{conj}

\begin{remark}
    See \cite[\S2]{DGKM2} for a discussion of the quantity $w(C^\tau)$, omitted from the present work because root numbers are not required in the Magma implementation for $L$-function computations.
\end{remark}

\begin{remark}
    In this section we have held our focus on $\ell$-adic cohomology, but we note that we can and will associate to $C^\tau$ other cohomology theories in much the same way.
\end{remark}

\section{The algorithm}\label{sec:algo}
We now describe an algorithm for computing the Euler factors at good primes of $L(C^\tau,s)$. In light of \eqref{eqn:trace}, the important step is computing the traces $\Tr(\Frob_\pp\cdot g\mid H^1_\ell(C^\tau)^{I_\pp})$ for $g\in G$; the novel aspect of our algorithm is the construction of auxiliary curves, point counts on which will yield such traces. The advantage of reducing the problem to point counts on curves---as opposed to relying, for example, on Gr{\"o}bner basis techniques to count fixed points---is that there are many efficient algorithms implemented in Magma for these computations.

\begin{algorithm}\label{algo}{\hspace{2cm}}\\
\textbf{Input:} A (model of a) curve $C$ with an action by automorphisms of $G$, a representation $\tau$ of $G$, and a prime $\pp$ of good reduction with $\pp\nmid|G|$. \\
\textbf{Output:} The local polynomial of $L(C^\tau,s)$ at $\pp$.
\begin{enumerate}
    \item For each $g\in G$ compute $\Tr((\Frob_\pp\cdot g)^i \mid H^1_\ell(C))$ for $i\in\{1,\hdots,g(C)\}$. We explain how to do this with $i=1$, the rest is the same after base change: \begin{enumerate}
        \item Obtain a model for the quotient curve $C_g:=C/\langle g\rangle$ over $k_\pp$.
        \item Construct the following $C_d\times C_d$-diagram of function fields, where we identify $k_\pp$ with $\FF_q$ and write $d$ for the order of $g$:
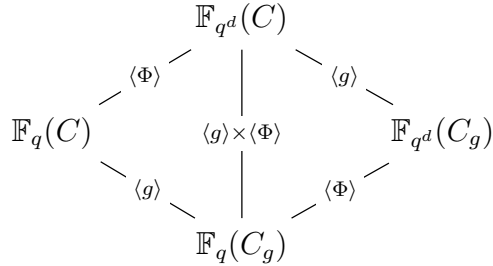
        \begin{equation}\label{eqn:curve_diagram}\begin{tikzcd}
	& {\FF_{q^d}(C)} \\
	{\FF_q(C)} && {\FF_{q^d}(C_g)} \\
	& {\FF_q(C_g)}
	\arrow["{\langle\Phi\rangle}"{description}, no head, from=1-2, to=2-1]
	\arrow["{\langle g\rangle}"{description}, no head, from=1-2, to=2-3]
	\arrow["{\langle g\rangle\times \langle\Phi\rangle}"{description}, no head, from=1-2, to=3-2]
	\arrow["{\langle g\rangle}"{description}, no head, from=2-1, to=3-2]
	\arrow["{\langle\Phi\rangle}"{description}, no head, from=2-3, to=3-2]
\end{tikzcd}\end{equation}
    Write $\tilde{C}_g$ for the quotient curve $C_{\FF_{q^d}}/\langle g\cdot \Phi\rangle$, which is defined over $\FF_q$. By Lemma \ref{lem:aux_curve} below, we have
    \begin{equation}\label{eqn:aux_trace}\tag{$\dagger$}\Tr(\Frob_\pp\cdot g\mid H^1_\ell(C)) = 1+q-\#(\tilde{C}_g(\FF_q)),\end{equation} so it suffices to obtain equations for this curve; even if we find a singular curve birational to $\tilde{C}_g$, we can adjust this point count for singularities. An equation can be found as follows:
    \begin{enumerate}
        \item Write $r$ for the minimum degree of an extension of $\FF_q$ which contains the $d$-th roots of unity and base change all curves to a degree $r$ extension.
        \item Write $\FF_{q^r}(C)/\FF_{q^r}(C_g)$ as a Kummer extension with minimal polynomial $X^d-R$.
        \item Compute the function field $\FF_{q^r}(C_g)[X]/(uX^d - R)$ of $(\tilde{C}_g)_{\FF_{q^r}}$, where $u$ is as in Lemma \ref{lem:aux_eqn} below.
        \item By taking the trace of a $d$-th root of $R/u$ and computing the minimal polynomial thereof, obtain an equation for $\tilde{C}_g$ over $\FF_q$.
    \end{enumerate}
    \end{enumerate}
    \item Compute $\Tr(\Frob_\pp^i\mid H^1_\ell(C^\tau))$ for each $i$ using \eqref{eqn:trace}.
    \item Output $\det(1-\Frob_\pp^{-1} T \mid H^1_\ell(C^\tau))$, which can now be computed via Newton's identities.
\end{enumerate}
\end{algorithm}

\begin{remark}
    The requirement that $\pp$ be of good reduction is necessary, although the method described above would compute the correct factor of the $L$-polynomial of the smooth normalisation of the special fibre of a regular model of $C_{K_\pp}$ at a bad prime $\pp$. To extract the full Euler factor, one must compute the action of $G$ on the special fibre---in practice either the conductor exponent is small and this factor suffices, or it is easier to guess the Euler factor.
\end{remark}

\begin{remark}
    To extend Algorithm \ref{algo} to primes $\pp\mid |G|$, one needs to work with Artin--Schreier theory instead of Kummer theory. In principle this is not a problem, but it is harder to extract equations for the auxiliary curves as in step (1)(b)(iii). In practice, the order of $G$ will have few divisors and again it will be easier to guess Euler factors at these primes.
\end{remark}

It remains to justify \eqref{eqn:aux_trace} and step (1)(b)(iii), which we do with the following lemmata:

\begin{lemma}\label{lem:aux_curve}
    With the set-up as in Algorithm \ref{algo}, equation \eqref{eqn:aux_trace} holds. That is,
    \[\Tr(\Frob_\pp\cdot g\mid H^1_\ell(C)) = 1+q-\#\tilde{C}_g(\FF_q).\]
\end{lemma}

\begin{proof}
    By Lefschetz's trace formula, it suffices to show that the number of fixed points of $g\cdot\Phi$ acting on $C_{\FF_{q^d}}$ is equal to $\#\tilde{C}_g(\FF_q)$. This holds because $\FF_{q^d}(C)/\FF_q(\tilde{C}_g)$ is an extension of the constant field; $g\cdot\Phi$ is the Frobenius endomorphism of $\tilde{C}_g$.
\end{proof}

\begin{lemma}\label{lem:aux_eqn}
    Keep the set-up as in Algorithm \ref{algo}, recalling in particular that $\pp\nmid |G|$ is coprime to $d$. Firstly, write $K(\zeta_d)(C)/K(\zeta_d)(C_g)$ as a Kummer extension with minimal polynomial $X_0^d-R_0$ so that $\FF_{q^r}(C)/\FF_{q^r}(C_g)$ is given by a Kummer extension with minimal polynomial $X^d-R$. Then $\FF_{q^r}(\tilde{C}_g)/\FF_{q^r}(C_g)$ can be written as a Kummer extension with minimal polynomial 
    \[uX^d-R, \hspace{0.5cm}\text{where}\hspace{0.5cm} u^{(q^r-1)/d}=z^{-1}.\]
    Here $z$ is a $d$-th root of unity constructed in the following way: choose a prime $\PP$ of $K(\zeta_d)$ above $\pp$, identifying the residue field of $\PP$ with $\FF_{q^r}$, and set $z$ to be the reduction modulo $\PP$ of $g(X_0)/X_0$.
\end{lemma}

\begin{proof}
    Suppose without loss of generality that $C_g$ is given by a (possibly singular) affine equation $f(Y,Z)=0$ over $\FF_q$. The fixed points of $g\cdot\Phi$ acting on $C_{\FF_{q^r}}$ are those with $(z\cdot X^{q^r},Y^{q^r},Z^{q^r})=(X,Y,Z)$, i.e.\ $Y,Z\in\FF_{q^r}$ and $X$ satisfying $z\cdot X^{q^r}-X=0$. The solutions to this equation in $X$ lie in $\FF_{q^{dr}}$; indeed if $u_0$ is a solution to $X^{q^r-1}=z^{-1}$, then the full set of solutions is $u_0\cdot \FF_{q^r}$.
    
    Substituting into the equations $\{X^d-R,f(Y,Z)=0\}$ for $C_{\FF_{q^r}}$, we obtain equations $\{u X^d-R, f(Y,Z)=0\}$ for $(\tilde{C}_g)_{\FF_{q^r}}$, noting that $u=u_0^d\in \FF_{q^r}$.
\end{proof}

\begin{remark}
    Our implementations of Algorithm \ref{algo} are in Magma, within which one can work with general motivic $L$-functions using their implementation of a version of T. Dokchitser's $L$-function calculator. This works by estimating the error introduced by truncating the Dirichlet series for a given $L$-function at $N\in\NN$. For a set precision, one can ask Magma for an approximation of the number of coefficients needed for accurate calculation using the command \texttt{LCfRequired}, and so in practice it suffices to work with truncated local polynomials rather than computing the whole thing as described above.
\end{remark}

\section{Endomorphisms of Jacobians}
Here we recap some background on endomorphisms of Jacobians, with the point being that the framework of motivic pieces of curves---under certain conditions---can be used to describe factorisations of $L$-functions of curves whose Jacobians have extra \textit{endomorphisms}, rather than automorphisms.

\subsection{Endomorphisms of Hecke type}\label{sec:HeckeType}
In \cite{Ellenberg01}, Ellenberg introduces the notion of endomorphism algebras of Hecke type in order to construct families of curves whose Jacobians have real multiplication, i.e.\ curves $C$ of genus $g$ for which there is a totally real field $F$ of degree $g$ and an injection $F\hookrightarrow \End(\Jac(C))\otimes_\ZZ\QQ$.

\begin{definition}\label{def:HeckeType}
    Consider a curve $C$ over $k$ equipped with a cover $\pi\colon C\to B$ of another curve. Say the Galois closure of the field extension $k(C)/k(B)$ has Galois group $G$. This closure corresponds to a curve $D$, which admits an action by $G$ of automorphisms. Write $H\le G$ for the subgroup for which $C = D/H$. There is a map $\QQ[H\backslash G/H]\hookrightarrow \End(\Jac(C))\otimes_\ZZ\QQ$---say the image is an \textit{endomorphism algebra of Hecke type} and that a curve $C$ has extra endomorphisms of Hecke type if there exists a cover $\pi$ such that this image is non-trivial.
\end{definition}

\begin{remark}
    Given a curve $C$ whose Jacobian has extra endomorphisms of Hecke type, we obtain a factorisation of $L(C,s)$ corresponding to the factorisation of $L(D,s)$ under the action of $G$, in the notation of Definition \ref{def:HeckeType}.
\end{remark}

\begin{remark}
    Ellenberg notes that quotients of $\QQ[H\backslash G/H]$ are central simple algebras over abelian number fields and so, in particular, not all non-trivial endomorphism algebras can arise in this way. Whether there are other obstructions to endomorphism algebras being of Hecke type remains an open question.
\end{remark}

\subsection{{$\GL_2$}-type and modularity}\label{sec:GL2}
We will also be interested in Jacobians of $\GL_2$-type. Being of $\GL_2$-type is a condition on the endomorphism algebra, which determines precisely when an $L$-function of an abelian variety decomposes as a product of $L$-functions of modular forms.

\begin{definition}
    An abelian variety $A/K$ of dimension $g$ is of $\GL_2$-type if $V_\ell(A)$ is a $G_K$-representation taking values in $\GL(2,k\otimes\QQ_\ell)$, where $k\hookrightarrow \End(A)\otimes_\ZZ \QQ$ is a number field of degree $g$.
\end{definition}

\begin{theorem}\label{thm:GL2}
    Given a simple abelian variety $A$ over $\QQ$, there exists $N\in\NN$ such that $A$ is a quotient of $J_1(N)$ if and only if $A$ is of $\GL_2$-type.
\end{theorem}

\begin{proof}
    In \cite{Ribet2004} Ribet shows that this is a consequence of Serre's Conjecture, a proof of which is given in \cite{KWSerreConj}.
\end{proof}

\begin{remark}
    A simple abelian variety $A$ being a quotient of $J_1(N)$ means precisely that there exists a newform $f$ of conductor $N$ such that $L(A,s)=\prod_\tau L(f^\tau,s)$, where this product is over the Galois conjugates of $f$.
\end{remark}

\subsection{An example}\label{sec:dihedral}
It was pointed out to the author by T. Dokchitser that the two-parameter family
\[f_{t,y}(x)=x^5 + yx^4 + tx^3 + (-2t+y^2+3y+5)x^2 + (t+y)x + (y+3)\]
consists of polynomials with $D_{10}$-Galois group, where $D_{10}$ is the dihedral group of order $10$. The unique quadratic subfield of the Galois closure of $f_{t,y}(x)$ is generated by $\sqrt{-\delta_{t,y}}$, where
\begin{eqnarray}
    \delta_{t,y}=4y^5 + 31y^4 &+& (118-26t)y^3 + (245-122t-t^2)y^2\\
    &+&(350-270t+36t^2)y+(375-300t+80t^2+4t^3).
\end{eqnarray}
We write $f_t(x,y)=f_{t,y}(x)$ and consider this as a one-parameter family of genus $2$ curves $C_t: f_t(x,y)=0$, noting that a Weierstrass model is given by
\begin{equation}\label{eqn:W_model}
    C_t: \hspace{0.2cm} y^2 = x^6-2x^5+x^4+(6-4t)x^3+2x^2+4x+1.
\end{equation} 
We also write $\delta_t(y)=\delta_{t,y}$. We obtain the following $G=D_{10}$-diagram of curves:
\[\begin{tikzcd}
& {X_t:f_t(x,y)=\Delta^2 + \delta_t(y) = 0} \\
{C_t: f_t(x,y)=0} \\
&& {D_t: \Delta^2 +\delta_t(y)=0 } \\
& {\mathbb{P}^1_y}
\arrow[from=1-2, to=2-1]
\arrow[from=1-2, to=3-3]
\arrow[from=2-1, to=4-2]
\arrow[from=3-3, to=4-2]
\end{tikzcd}\]

Via \cite[Theorem A]{KaniRosen} and the existence of the Brauer relation $\one-2C_2-C_5+G$ in $G$, we obtain a factorisation $L(X_t,s)=L(C_t,s)^2\cdot L(D_t,s)$. By comparing with the decomposition under the action of $G$, writing $\rho$ and $\rho'$ for the conjugate two-dimensional irreducible representations, we find
\begin{equation}\label{eqn:Ct_fact}
    L(C_t,s)=L(X_t^\rho,s)\cdot L(X_t^{\rho'},s).
\end{equation}
Concomitantly, there is another decomposition of $L(C_t,s)$ afforded by Theorem \ref{thm:GL2}:

\begin{lemma}\label{lem:GL2}
    The $L$-function $L(C_t,s)$ factors as a product of $L$-functions of classical modular forms with $q$-expansion coefficients in $\QQ(\sqrt{5})$.
\end{lemma}

\begin{proof}
    It suffices to show that $\QQ(\sqrt{5})\hookrightarrow\End(\Jac(C_t))\otimes_\ZZ\QQ$, so that $\Jac(C_t)$ is of $\GL_2$-type. To see this, note that $\pi_*\circ g_*\circ \pi^*$ acts as $(1+\sqrt{5})/2$ on $\Jac(C_t)$, where $\pi: X_t \to C_t$ is the quotient map and $g$ is an automorphism of order $5$ acting on $X_t$.
\end{proof}

It turns out that the factorisation \eqref{eqn:Ct_fact} and the factorisation described in Lemma \ref{lem:GL2} are the same, the salient point being that the endomorphism in the proof of Lemma \ref{lem:GL2} is of Hecke type.

\begin{proposition}\label{prop:mod_form}
    $L(X_t^\rho,s)$ is the $L$-function of a classical newform.
\end{proposition}

\begin{proof}
    Firstly, the decomposition of $L(C_t,s)$ into $L$-functions of modular forms is precisely that coming from the action of $\QQ(\sqrt{5})\hookrightarrow\End(\Jac(C_t))\otimes\QQ$. By Remark \ref{rmk:l_func_decomp}, in the notation of which we have $A_{\hat\rho}=\Jac(C_t)$, this is the same decomposition coming from the action of $D_{10}$ on $X_t$.
\end{proof}

\begin{example}\label{ex:X0(23)}
    The curve $C_2$ as in \eqref{eqn:W_model} has LMFDB (\cite{lmfdb}) label \href{https://www.lmfdb.org/Genus2Curve/Q/529/a/529/1}{\texttt{529.a.529.1}}, from which we see that its Jacobian is isogenous to that of the modular curve $X_0(23)$. We know that $L(X_0(23),s)$ factors as the product of the $L$-functions of the two newforms of level $23$ and weight $2$, whose $L$-functions can also be found on the LMFDB or via Magma's in-built $L$-series. We compute the Euler factors of $L(X_2^\tau,s)$ and $L(X_2^{\tau'},s)$ for each prime up to $10,000$ and confirm that these match with those stored in Magma (see \cite[\texttt{Dihedral.m}]{MotivicPiecesCode}). In light of Proposition \ref{prop:mod_form}, this provides a helpful sanity check for the validity of Algorithm \ref{algo}. 
\end{example}

\begin{remark}
    Example \ref{ex:X0(23)} is the only instance discussed herein where we must implement an instance of explicit Kummer theory as in step 1(b)(ii) of Algorithm \ref{algo}. To do this, we use the results of \cite{DummitQuintics} and an adaptation of the associated Mathematica \cite{Mathematica} script available at \cite{QuinticsCode}.
\end{remark}

In general, then, the algorithm we will describe in \S\ref{sec:algo} gives a method to explicitly recover the factorisation of a Jacobian of $\GL_2$-type whose endomorphisms are of Hecke type into a product of $L$-functions of newforms, provided that one can exhibit explicit covers as in Definition \ref{def:HeckeType}.

\section{Periods}\label{sec:periods}
In this section we discuss the conjectural link between periods of motives and $L$-values, before explaining theoretically how to compute periods of motivic pieces in general and how to practically do so for pieces of superelliptic curves in particular.

\subsection{Motives}\label{sec:motives}
We begin with a brief discussion of motives from the point of view of realisation data, which sets us up for the statement of Deligne's Period Conjecture.

Let $K$ and $F$ be number fields. For the purposes of this section, we view these abstractly, rather than having in mind an embedding into $\CC$. For us, a $K$-motive $M$ with coefficients in $F$ of dimension $d$ and weight $w$ is a collection of \textit{realisation data} $\{H_B(M),H_{\dR}(M),(H_\lambda(M))_{\lambda}, c_{B,\dR}, c_{\ell,B}, c_{\ell,\dR}\}$ as follows (e.g.\ \cite[\S2.2--2.3]{EvansThesis}):

\begin{itemize}
    \item The \textit{Betti realisation}, $H_B(M)$, is a $d$-dimensional $F$-vector space with Hodge filtration by free $F\otimes\CC$-modules \[H_B(M)\otimes \CC =\bigoplus_{i+j=w}H^{(i,j)}(M),\]
    equipped with an action of complex conjugation exchanging $H^{(i,j)}(M)$ and $H^{(j,i)}(M)$ for each real place of $K$.
    \item The \textit{de Rham realisation}, $H_{\dR}(M)$, is a $d$-dimensional $F$-vector space with an exhaustive decreasing filtration \[\{\Fil^kH_\text{dR}(M)\}_{k\in\ZZ}.\] 
    \item The \textit{$\lambda$-adic realisations}, $H_\lambda(M)$, indexed by proper prime ideals $\lambda$ of $F$, are $d$-dimensional $F_\lambda$-vector spaces equipped with a continuous homomorphism
    \[ \rho_\lambda: G_K \to \GL(H_\lambda(M)),\]
    and, for each prime $\pp$ of $K$ coprime to $\lambda$ of $F$, a local polynomial
    \[ L_\pp(M,T) = \det(1-T\Frob_\pp^{-1} \midx H_\lambda(M)^{I_\pp})\in F[T]\]
    independent of $\lambda$. There is a finite set $S=S(M)$ of primes $\pp$ of $K$ such that $\rho_\lambda(I_\pp)=1$ for $\pp\not\in S$, and that for $\pp\not\in S$, fixing an embedding $F\hookrightarrow\CC$, the complex roots of $L_\pp(M,T)$ have absolute value $N(\pp)^{-w/2}$.
    \item $c_{B,\dR}$ is an $F\otimes_\QQ\CC$-module isomorphism $H_B(M)\otimes_\QQ \CC\xrightarrow{\sim} H_{\dR}(M)\otimes_\QQ \CC$ respecting the action of complex conjugation and such that, for all $k$,
    \[ c_{B,\dR}\left(\bigoplus_{i+j=w,\hspace{0.1cm}i\ge k}H^{(i,j)}(M)\right)=\Fil^k H_{\dR}(M)\otimes_\QQ\CC. \]
    \item $c_{\ell,B}$ and $c_{\ell,\dR}$ are comparison isomorphisms relating $\oplus_{\lambda\mid\ell}H_\lambda(M)$ to $H_B$ and $H_{\dR}$, respectively. We omit explicit descriptions as we will not make use of these.
\end{itemize}

To the above data, we also associate an $L$-function.

\begin{definition}\label{defn:motLfn}
    Given a $K$-motive $M$ with coefficients in $F$, its $L$-function (or tuple of $L$-functions), converging in the right half-plane $\Re(s)>1+w(M)/2$
    \[ L(M,s)=(L(\sigma,M,s))_{\sigma\in\Sigma_F},\]
    where
    \[L(\sigma,M,s) = \prod_{(0)\ne\pp\triangleleftneq\OO_K} L_\pp^\sigma(M,N(\pp)^{-s})^{-1}\]
    and $L_\pp^\sigma$ is the polynomial over $\CC$ obtained by embedding the coefficients of $L_\pp$ via $\sigma$.
\end{definition}

\begin{remark}
    When we already have in mind an embedding $\sigma\colon F\hookrightarrow\CC$, we may write $L(M,s)$ for $L(\sigma,M,s)$.
\end{remark}

We note the existence of the \textit{Tate motive}, twisting a motive $M$ by which corresponds to a shift on the level of $L$-functions.

\begin{example}
    The Tate motive, $\QQ(1)$, is a $\QQ$-motive with coefficients in $\QQ$ of dimension $1$ and weight $-2$. Given a motive $M$, we obtain the $n$-th Tate twist $M(n)$ by tensoring with $\QQ(1)^{\otimes n}$ on the level of realisation data. The effect of this on the $L$-function is simply
    \[ L(M(n),s)=L(M,s+n).\] 
\end{example}

For this article, the example to have in mind is the motive associated to the first cohomologies of an abelian variety, whose $L$-function is the usual $L$-function associated to an abelian variety.

\begin{example}\label{ex:AbVarResData}
    Associated to the first cohomologies of a $g$-dimensional abelian variety $A$ over $K$ is a $K$-motive with coefficients in $\QQ$ of dimension $2g$ and weight $1$, denoted $h^1(A)$. The Betti and de Rham realisations are the first singular and de Rham cohomologies of $A$, respectively, and the $\ell$-adic realisation is the first {\'e}tale cohomology, or $(T_\ell(A)\otimes_{\ZZ_{\ell}}\QQ_\ell)^*$. Recall that BSD concerns the value $L(A,1)$, so it is pertinent to note that $h^1(A)(1)$ is isomorphic to the dual of $h^1(A)$.
\end{example}

\begin{remark}
    In the case of the Jacobian of a curve with an action by $G$ of automorphisms, decomposing the data of Example \ref{ex:AbVarResData} under the inherited $G$-action (as in Definition \ref{defn:motivic_piece}) gives us the realisation data associated to the motivic pieces $C^\tau$. Then $L(C^\tau,s)$ as in Definition \ref{defn:motLfn} agrees with the $L$-function of Definition \ref{defn:l_func}.
\end{remark}

\subsection{Deligne's Period Conjecture}

\begin{definition}\label{defn:del_period}
    Consider a $\QQ$-motive $M$ with coefficients in $F$ and odd weight $w$. Let $H_B(M)^+$ be the $+1$-eigenspace of complex conjugation, and define
    \[\alpha_{M}^+: H_B^+(M)\otimes\CC \hookrightarrow H_B(M)\otimes\CC \xrightarrow{c_{B,\dR}} H_{\text{dR}}(M)\otimes\CC\twoheadrightarrow H_\text{dR}^+(M)\otimes\CC,\]
    where the last map is the natural quotient map onto 
    \[H^+_{\dR}(M)=H_{\dR}(M)/\Fil^{1+\lfloor w(M)/2\rfloor}H_{\dR}(M).\] 
    In the case that $\alpha_M^+$ is an isomorphism, denote by $c^+(M,0)\in(F\otimes_\QQ \CC)^\times$ a representative of the class of $\det(\alpha^+_M)$ in $(F\otimes_\QQ \CC)^\times/F^\times$, computed with respect to $F$-bases of $H_B$ and $H_{\dR}$. We view this as a tuple $(c^+(\sigma,M,0))_{\sigma\in\Sigma_F}\in(\CC^\times/F^\times)^{\Sigma_F}$ indexed by the embeddings of $F$. Moreover, we write $c^+(M,s_0)$ for $c^+(M(s_0),0)$ and call this the \textit{Deligne period} of $M$ at $s_0$. 
\end{definition}

\begin{remark}
    There is an analogous definition in the case that $w$ is even, but for our purposes this is not required.
\end{remark}

As mentioned above, Deligne conjectures that these periods describe the `irrational parts' of $L$-values. To state this, we require one further definition.

\begin{definition}[Critical value]
    Given a $K$-motive $M$ with coefficients in $F$, say $s_0$ is a critical value for $M$ if neither $L_\infty(M,s)$ nor $L_\infty(M^*(1),s)$ has a pole at $s_0$.
\end{definition}

\begin{conj}[{= \cite[Conjecture 2.8]{Deligne79}}]\label{conj:Deligne}
    Let $M$ be a $\QQ$-motive with coefficients in $F$ and odd weight, and let $s_0\in\ZZ$ be a critical value for $M$. Then $\alpha_M^+$ is an isomorphism and there exists $\mathcal{L}(M,s_0)\in F$ such that 
    \begin{equation}\label{eqn:DelignePeriod}
        L(M,s_0)=c^+(M,s_0) \cdot \mathcal{L}(M,s_0).
    \end{equation}
\end{conj}

\begin{remark}
    Conjecture \ref{conj:Deligne} depends implicitly on the conjectural analytic continuation of $L(M,s)$. We take this as a standing assumption, and in practice use Magma to numerically verify the functional equation before computing any $L$-values.
\end{remark}

Conjecture \ref{conj:Deligne} extends to motives over any number field via Weil restriction, under which $L$-functions are invariant.

\begin{example}\label{ex:AV_periods}
    Fix an abelian variety over a number field, $A/K$. (Alternatively, a curve $C/K$, seeing as $C$ and $\Jac(C)$ have the same first cohomologies.) BSD implicitly predicts
    \begin{equation}\label{eqn:BSD_period}
    \frac{L(A,1)}{|\Delta_K|^{\dim A/2}\prod_{\nu\mid\infty}|\Omega_{A^\nu}|} \in \QQ.
    \end{equation}
    Here, this product is over the infinite places of $K$ and $\Omega_{A^\nu}$ is defined as follows: fix a $\QQ$-basis $\mathcal{W}$ for the holomorphic differentials $H^0(C,\Omega^1_C)$ and, for a given $\nu$, choose a basis $\Gamma_\nu$ for the singular homology group $H_1(A^\nu(\CC),\ZZ)$ such that, if $\nu$ is a real place, then there exists some $\Gamma_\nu^+\subset\Gamma_\nu$ which is a basis for $H_1(A^\nu(\RR),\ZZ)$.
    \begin{equation}\label{eqn:BSD_period_defn}\Omega_{A^\nu} = \begin{cases}
        \det\left(\int_\gamma\omega\right)_{\omega\in\mathcal{W},\gamma\in\Gamma_\nu^+} & \text{ if }\nu\text{ is a real place} \\
        \det\left(\int_\gamma\omega\right)_{\omega\in\mathcal{W}\cup\bar{\mathcal{W}},\gamma\in\Gamma_\nu} & \text{ if }\nu\text{ is a complex place.}
    \end{cases}\end{equation}
    That this is consistent with Conjecture \ref{conj:Deligne} can be seen from the formulation of BSD for abelian varieties over number fields in \cite{Gross1982}, applying restriction of scalars and referring to \cite[\S4]{Deligne79}---the key observation being that $h^1(A)(1)$ is isomorphic to the dual of $h^1(A)$, as mentioned in Example \ref{ex:AbVarResData}.

    We give a brief explanation of the provenance of the period term in \eqref{eqn:BSD_period} from Deligne's perspective, focussing for simplicity on the case that $K$ is totally real. Firstly, we must pass to the restriction of scalars $R=\Res_{K/\QQ}A$, which is an abelian variety over $\QQ$. Then $H_B^+(h^1(R))$ has dual $H_1(R(\RR),\ZZ)$, which we identify with $\prod_{\nu\mid\infty} H_1(A^\nu(\RR),\ZZ)$, using $R(\RR)=A(K\otimes_\QQ\RR)\cong\prod_{\nu} A^\nu(\RR)$. We take a basis $\Gamma_\nu^+$ for each place as above. On the other hand, for a basis $\{k_i\}$ for $K/\QQ$, $H^+_{\dR}(h^1(R))=H^0(R,\Omega^1_R)$ has basis $\{ k_i\cdot \omega \mid \omega\in H^0(A,\Omega^1_A), i\in\{1,\hdots,[K:\QQ]\}\}$.
    
    Labelling the infinite places of $K$ as $\nu_j$, the period matrix of $R$ is now made up of $\dim A\times \dim A$ blocks with the $(i,j)$-th block given by
    \[ \left( \nu_j(k_i)\int_\gamma\omega \right)_{\omega\in\mathcal{W},\gamma\in\Gamma_j}=\nu_j(k_i)\cdot P_j,\hspace{0.2cm}\text{ where }\hspace{0.2cm} P_j = \left(\int_\gamma\omega\right)_{\omega\in\mathcal{W},\gamma\in\Gamma_j}, \]
    using the classical result that $\omega\mapsto[\gamma\mapsto \int_\gamma\omega]$ defines an isomorphism from de Rham cohomology to the dual of singular cohomology. In particular, this matrix is equal to
    \[ \left(\left(\nu_j(k_i)\right)_{i,j}\otimes I_{g_C}\right) \cdot \text{diag}(P_j)_j, \] 
    where $\otimes$ denotes the Kronecker product. The determinant of $\left(\nu_j(k_i)\right)_{i,j}\otimes I_{g_C}$---up to rationals---is $(\sqrt{\Delta_{K/\QQ}})^{g_C}$ by definition of the discriminant.    
    
    If $A$ has coefficients in $F$, to compute the period $c^+(h^1(A),1)\in (F\otimes_\QQ\CC)^\times$ as in \eqref{eqn:DelignePeriod}, one must then choose $F$-bases for the de Rham cohomology and singular homologies and compute the determinants as in \eqref{eqn:BSD_period_defn} with respect to these bases. The effect of forgetting this additional structure is to take the norm map $F\otimes_\QQ\CC\to \CC$, in which case our periods and $L$-values are multiplied together and we recover the prediction \eqref{eqn:BSD_period} from Conjecture \ref{conj:Deligne}.
\end{example}

We now return to the context of motivic pieces of curves. Given $C/K$, $G$, $\tau$ as usual, write $\hat{\tau}$ for the rational representation $\bigoplus_{\sigma\in\Gal(\QQ(\tau)/\QQ)}\sigma(\tau)$. Recall from Remark \ref{rmk:l_func_decomp} that the piece $C^{\hat{\tau}}$ corresponds to the motive of an abelian variety, $A_{\hat{\tau}}$, a quotient of $\Jac(C)$. This abelian variety has coefficients in $\QQ(\tau)$, with the corresponding $L$-values being associated to the pieces $C^\tau$ under different embeddings $\QQ(\tau)\hookrightarrow\CC$ or, equivalently, the Galois conjugate pieces under a fixed embedding. Applying Conjecture \ref{conj:Deligne} to $A_{\hat\tau}$ then gives a prediction for the $L$-values $L(C^{\sigma(\tau)},1)$. 

Because the $\QQ(\tau)$-coefficients are coming from the $G$-action, this is equivalent to taking the determinant of $\tau$-part of the period map from $H^0(\Res_{K/\QQ}C,\Omega^1_{\Res_{K/\QQ}C})$ to the dual of $H_1(\Res_{K/\QQ}C(\RR),\ZZ)$, with different embeddings $\QQ(\tau)\hookrightarrow\CC$ corresponding to the periods of Galois-conjugate pieces.

\subsection{Superelliptic curves}\label{sec:super_periods}
Molin and Neurohr explicitly describe methods to compute period matrices associated to superelliptic curves over $\CC$ to high precision in \cite{SuperellipticPeriods}, with implementations of their work available at \cite{SuperellipticPeriodsCode}. The methods they use do not give the full $2g\times2g$-matrices as in the complex case of \eqref{eqn:BSD_period_defn}, but rather the $g\times 2g$-matrix
\begin{equation}\label{eqn:partial_mat}
    P_{\mathcal{W},\Gamma} = \left(\int_\gamma\omega\right)_{\omega\in\mathcal{W},\gamma\in \Gamma}
\end{equation}
for specific choices $\mathcal{W}$ and $\Gamma$ of bases of $H^0(C,\Omega^1_C)$ and $H_1(C,\ZZ)$, respectively. We explain how to extract periods of motivic pieces of superelliptic curves from their work.

Firstly, we fix a superelliptic curve $C/K: y^m = f(x)$ and an embedding $\sigma: K\hookrightarrow\CC$, where $f$ is separable and $K$ is a number field containing the $m$-th roots of unity. Write $C_\sigma$ for the complex curve corresponding to this embedding, and take $n=\deg f$, $\delta=\gcd(m,n)$. We consider the group of automorphisms $G\cong C_m$ generated by $\alpha: y\mapsto \zeta_m\cdot y$.

\begin{proposition}[{= \cite[Proposition 3.8]{SuperellipticPeriods}}]\label{prop:holom_diffs}
The set 
\begin{equation}\label{eqn:holom_diffs}
    \mathcal{W}=\left\{ \omega_{i,j}:=\frac{x^{i-1}\emph{d}x}{y^j} \mid 1\le i\le n-1, \hspace{0.2cm} 1\le j \le m-1,\hspace{0.2cm} -mi+jn-\delta\ge 0 \right\}
\end{equation}
forms a basis for $H^0(C,\Omega^1_{C})$.
\end{proposition}

We will write $\mathcal{W}_\sigma=\{\sigma(\omega)\mid\omega\in\mathcal{W}\}$. The action of $G$ on $H^0(C_\sigma,\Omega^1_{C_\sigma})$ is thus clear; $\alpha$ acts on $\omega_{i,j}$ by $\sigma(\zeta_m)^{-j}$ and $\mathcal{W}_\sigma$ is already an eigenbasis for $\alpha$. The story for singular homology is less straightforward. To construct a basis for $H_1(C_\sigma,\ZZ)$, for each pair $(a,b)$ of branch points of the rational map $x$, a branch $y_{a,b}$ of $C$ is defined (see \cite[p.\ 6]{SuperellipticPeriods} for details) and hence oriented paths
\begin{equation}\label{eqn:branches}
\gamma_{[a,b]}^{(\ell)}=\{(x,\sigma(\zeta_m)^{\ell}\cdot y_{a,b}(x)) \mid x\in [a,b]\},\hspace{0.5cm} \ell\in\ZZ/m\ZZ
\end{equation}
are obtained. Here $[a,b]$ is the oriented line segment joining $a$ and $b$ in the complex plane. We obtain loops
\[ \gamma_{a,b}^{(\ell)} = \gamma_{[a,b]}^{(\ell)} \cup \gamma_{[b,a]}^{(\ell)} \]
for each pair of branch points. Write $X$ for the complete graph whose vertices are branch points of $x$, and $E$ for the edges of any choice of spanning tree of $X$. Writing $\gamma_{e}^{(\ell)}$ for $\gamma_{a,b}^{(\ell)}$, where $e$ is the edge between vertices $a,b\in X$, we are now able to describe generators for $H_1(C_\sigma,\ZZ)$.

\begin{theorem}[{= \cite[Theorem 3.6]{SuperellipticPeriods}}]\label{thm:loops}
    The set
    \begin{equation}\label{eqn:loops}
        \Gamma_\sigma = \{\gamma_e^{(\ell)} \mid 0\le \ell < m-1, e \in E\}
    \end{equation}
    forms a generating set for $H_1(C_\sigma,\ZZ)$.
\end{theorem}

\begin{remark}
    In general, $\Gamma_\sigma$ is not a basis for $H_1(C_\sigma,\ZZ)$, although this is the case when $\delta=1$.
\end{remark}

\begin{remark}\label{rmk:loop_rel}
    The reason for the strict inequality in \eqref{eqn:loops} is that $\prod_{\ell=0}^{m-1}\gamma_{e}^{(\ell)}=1$ in $H_1(C_\sigma,\ZZ)=\pi_1(C_\sigma)^{\text{ab}}$.
\end{remark}

Molin and Neurohr then describe and implement a method for computing $\int_\gamma\omega$ for $\omega\in\mathcal{W}_\sigma$ and $\gamma\in\Gamma_\sigma$. We explain how to use their implementation to compute Deligne periods in this setting, writing $R=\Res_{\QQ(\zeta_m)/\QQ}C$. 

Firstly, note that the graph $X$ as above exists independently of any embedding $\QQ(\zeta_m)\hookrightarrow\CC$ and select a set $E$ of edges of a spanning tree. For each such embedding $\sigma\in\Sigma_{\QQ(\zeta_m)}$, we see from \eqref{eqn:branches} that $\Gamma_{\sigma,E} := \{\gamma_e^{(0)} \mid e \in E\}$ generates $H_1(C_\sigma,\QQ)$ as a $\QQ(\zeta_m)$-vector space. We will write $\Gamma'_\sigma\subseteq \Gamma_{\sigma,E}$ for a $\QQ(\zeta_m)$-basis of $H^1(C_\sigma,\QQ)$. Taking the restriction of scalars and writing $\Sigma_{\QQ(\zeta_m)}^+$ for a set of coset representatives of $\Sigma_{\QQ(\zeta_m)}$ modulo complex conjugation---corresponding to infinite places of $\QQ(\zeta_m)$---we see that 
\[\cup_{\sigma\in\Sigma_{\QQ(\zeta_m)}^+}\{\gamma \mid \gamma\in\Gamma'_\sigma\}\] is a $\QQ(\zeta_m)$-basis for $H_1(R(\RR),\QQ)$ after identifying with $\prod_{\sigma\in\Sigma_{\QQ(\zeta_m)}^+}H_1(C_\sigma(\CC),\QQ)$. We note also that $H^0(R,\Omega^1_R)$ has $\cup_{\sigma\in\Sigma_{\QQ(\zeta_m)}}\mathcal{W}_\sigma$ as a $\QQ(\zeta_m)$-basis. Therefore, we wish to compute the determinant of the period map
\[ \text{diag}\left( \left(\int_{\gamma} \omega \right)_{\omega\in \mathcal{W}_\sigma\cup\bar{\mathcal{W}}_{\sigma}, \gamma\in \Gamma'_\sigma}\right)_{\sigma\in\Sigma_{\QQ(\zeta_m)}^+}. \]
To take the determinant of the $\tau$-part of this map, write $\mathcal{W}_{\sigma,\tau}$ for $\{\omega \mid \omega\in\mathcal{W}_\sigma\cup\bar{\mathcal{W}}_\sigma, \alpha(\omega)/\omega=\tau(\alpha)\}$, $\pi_\tau$ for the $\tau$-projector in $\CC[G]$, and $\Gamma'_{\sigma,\tau}\subseteq \Gamma'_\sigma$ for a subset such that $\{\pi_\tau(\gamma)\mid\gamma\in\Gamma'_{\sigma,\tau}\}$ generates the $\tau$-component of $H_1(C_\sigma,\QQ)$. We need to compute the determinants of the matrices
\[ \left(\int_{\pi_\tau(\gamma)}\omega \right)_{\omega\in\mathcal{W}_{\sigma,\tau},\gamma\in\Gamma'_{\sigma,\tau}} = \left(\int_{\gamma}\omega \right)_{\omega\in\mathcal{W}_{\sigma,\tau},\gamma\in\Gamma'_{\sigma,\tau}} \]
as $\sigma$ varies over $\Sigma_{\QQ(\zeta_m)}^+$, where equality holds here because the comparison isomorphism is $G$-equivariant. We are viewing $\gamma\in\Gamma'_{\sigma,\tau}$ as an element of $H_1(R(\RR),\ZZ)$, on which complex conjugation acts trivially. In this context, we may replace $\int_\gamma\bar\omega$ by $\bar{\int_\gamma\omega}$. We therefore compute our Deligne periods as
\begin{equation}\label{eqn:piece_period}
   \Omega_{C^{\sigma(\tau)}} = \prod_{\sigma\in\Sigma_{\QQ(\zeta_m)}^+}\det
\begin{pmatrix}
P_{\sigma,\tau} \\
\hline \vspace{-0.4cm}\\ 
\bar{P_{\bar\sigma,\tau}}
\end{pmatrix}, \hspace{0.2cm} \text{ where }\hspace{0.2cm} P_{\sigma,\tau} = \left(\int_\gamma \omega\right)_{\omega\in\mathcal{W}_\sigma\cap\mathcal{W}_{\sigma,\tau},\gamma\in\Gamma'_{\sigma,\tau}}.
\end{equation}

\begin{remark}
    Note that there is no discriminant term appearing here as in Example \ref{ex:AV_periods}. One could state Conjecture \ref{conj:super_deligne} for superelliptic curves over finite extensions $K$ of $\QQ(\zeta_m)$, in which case an analogous term involving the relative discriminant $\Delta_{K/\QQ(\zeta_m)}$ would appear. In practice, working over proper extensions of $\QQ(\zeta_m)$ is computationally out-of-reach, so we are content with this simpler formulation. 
\end{remark}

Finally, Conjecture \ref{conj:Deligne} predicts:
\begin{conj}\label{conj:super_deligne}
    Let $C/\QQ(\zeta_m): y^m = f(x)$ be as above, choose $\tau$ a character of $C_m$. Fix an embedding $\QQ(\zeta_m)\hookrightarrow\CC$. There exists $\mathcal{L}(C^\tau)\in\QQ(\zeta_m)$ such that, for $\sigma\in\Gal(\QQ(\tau)/\QQ)$,
    \begin{equation}\label{eqn:super_deligne}
        L(C^{\sigma(\tau)},1) = \Omega_{C^{\sigma(\tau)}}\cdot \sigma(\mathcal{L}(C^\tau)).
    \end{equation}
\end{conj}

In the sequel we combine \cite{SuperellipticPeriodsCode} with our implementations of Algorithm \ref{algo} to numerically verify Conjecture \ref{conj:super_deligne} for the motivic pieces of several superelliptic curves over $\QQ(\zeta_3)$ and $\QQ(i)$.

\section{Computations}\label{sec:examples}
\subsection{Exponent three}\label{sec:exp_3}
Here we present some computations of $L$-functions, $L$-values and periods for superelliptic curves $C/\QQ(\zeta_3): y^3 = f(x)$ with $\deg f = 4$ and $f$ is separable. Such curves have genus $3$ and are known as Picard curves.

\begin{example}\label{ex:pic1}
    Consider the Picard curve
    \begin{equation}
        C/\QQ(\zeta_3): \hspace{0.2cm} y^3 = x^4+x^3+\zeta_3x^2+\zeta_3-2
    \end{equation}
    and write $\tau$ for a non-trivial character of $C_3$. We use an implementation of Algorithm \ref{algo} (see \cite[{\texttt{Superelliptic\_Exp3.m}}]{MotivicPiecesCode}) to compute Dirichlet series coefficients $L(C^\tau,s)$, with the caveat being that we guess the correct Euler factor at the unique prime $\pi$ above $3$ and the conductor exponent at $3$ from the conjectural functional equation. To guess the conductor, we assume that the bound
    \begin{equation}\label{eqn:CondBound}
        v_\pi(\mathcal{N}(C)) \le v_\pi(\Delta)
    \end{equation}
    is satisfied, thereby reducing to a finite check. Note that \eqref{eqn:CondBound} is not known to hold in this setting, but is generally expected to. 

    Taking $L_\pi(C^\tau,s)=1$ and $v_\pi(N(C^\tau))=12$, we find that $L(C^\tau,s)$ seems to satisfy its functional equation---the output of Magma's \texttt{CFENew}, computed to precision $15$, is less than $10^{-11}$ given $100,000$ Dirichlet series coefficients. Although we are primarily concerned with computing the $L$-functions of the pieces, we note that Magma estimates that a number of coefficients of the order $10^{9}$ would be needed to na{\"i}vely verify the Hasse--Weil Conjecture for $C$ to this same precision.

    We compute $L(C^\tau,1)\approx 1.84609002978320\hdots - i\cdot 1.54470301185825\hdots \ne0$, making this a candidate for non-trivial numerical verification of Conjecture \ref{conj:super_deligne}. We obtain the value $\Omega_{C^\tau}= -38.6473408677\hdots - i\cdot105.727015294\hdots$ as in \eqref{eqn:piece_period} via the method described in \S\ref{sec:super_periods}, and compute
    \begin{equation}\label{eqn:LValue1}
        \frac{L(C^\tau,1)}{\Omega_{C^\tau}}\approx -0.0185185185185\hdots + i\cdot0.0106916716517\hdots \approx \frac{\zeta_3-1}{81}.
    \end{equation}
    The error in the latter approximation is approximately $10^{-13}$---less than $81^{-6}$ and so vastly smaller than one would expect from an approximation with denominator of this size, and roughly to the same precision to which we computed the $L$-function.

    Because Galois acts as complex multiplication on $\QQ(\zeta_3)$, which respects the functional equation \eqref{eqn:func2}, we have both $L(C^{\bar\tau},1)=\bar{L(C^\tau,1)}$ and $\Omega_{C^{\bar\tau}}=\bar{\Omega_{C^\tau}}$. Hence the Galois equivariance part of Conjecture \ref{conj:super_deligne} would automatically follow if \eqref{eqn:LValue1} were a precise equality.
\end{example}

\begin{remark}
    The algorithm described in \cite{Lupoian2Tors} allows one to provably compute the conductor exponent at $3$, so in principle one needn't guess the conductor here. In practice this is implemented only for curves defined over $\QQ$, making guess-work the easier option, although there is no theoretical obstacle to extending the implementation to number fields.
\end{remark}

\begin{example}\label{ex:pic2}
    Now consider the curve
    \begin{equation}
        C/\QQ(\zeta_3): \hspace{0.2cm} y^3 = x^4- 3x^3+(1+\zeta_3)x.
    \end{equation}
    Using the same methods and notation as in Example \ref{ex:pic1}, we numerically verify the functional equation for $L(C^\tau,s)$ to precision $10^{-11}$ using approximately $360,000$ Dirichlet series coefficients, with $L_\pi(C^\tau,s)=1$ and $v_\pi(N(C^\tau))=18$. As before, we compute $L(C^\tau,1)\approx 2.29230558647\hdots + i\cdot0.0927222735424\hdots$ and $\Omega_{C^\tau}=-99.3426636471\hdots - i\cdot157.045532482\hdots$, giving
    \begin{equation}\label{eqn:LValue}
        \frac{L(C^\tau,1)}{\Omega_{C^\tau}}\approx -0.00617283950616\hdots + 0.0106916716517\hdots \approx \frac{\zeta_3}{81},
    \end{equation}
    where the error in the latter approximation is again smaller than the precision to which our $L$-function was computed.
\end{example}

\begin{remark}
    In both examples we have taken Picard curves $C: y^3 = f(x)$ where the coefficients of $f$ generate $\QQ(\zeta_3)$. If we were to take $C/\QQ(\zeta_3)$ to be the base change of a curve defined over $\QQ$, then we would simply find $L(C^\tau_{\QQ(\zeta_3)},s)=L(C,s)$.
\end{remark}

\subsection{Exponent four}\label{sec:exp_4}
We consider the family of genus $3$ curves of the form $C/\QQ(i): y^4 = f(x)$, where $\deg f=3$. Here, the situation is somewhat kinder computationally than the case of Picard curves over $\QQ(\zeta_3)$: fixing $\tau$ to be one of the two complex characters of $C_4$, we find that $C^\tau$ has dimension $2$ because there is a $2$-dimensional piece corresponding to the elliptic curve which arises as the quotient of $C$ by the action of $C_2\le C_4$. Moreover, because we expect a large contribution to the conductor from primes dividing the exponent, we also typically deal with smaller conductors than in the Picard curve case. Altogether, this means we are able to generate several examples (see \cite[{\texttt{Superelliptic\_Exp4.m}}]{MotivicPiecesCode}).

\renewcommand{\arraystretch}{1.5}
\begin{table}[ht]
    \centering
    \begin{tabular}{|c|c|c|c|c|}
        \hline
        $f(x)$ & $N(C^\tau)$ & $160\cdot L(C^\tau,1)/\Omega_{C^\tau}\approx$ \\
        \hline
        $x^3+ix^2-(1-i)x$ & $2^{17}\cdot 5$ & $3-i$ \\
        \hline
        $x^3 + (1+2i)x^2 - 2x$ & $2^{14}\cdot41$ & $2+4i$ \\
        \hline
        $x^3+ix$ & $2^{20}$ & $1-2i$ \\
        \hline
        $x^3-x^2+ix$ & $2^{16}\cdot 17$ & $2+i$ \\
        \hline
        $x^3 + (3+2i)x^2 + 2x$ & $2^{14}\cdot3^2\cdot17$ & $-2-6i$ \\
        \hline
        $x^3-3ix^2-(3-i)x$ & $2^{17}\cdot 5^2$ & $-5i$ \\  
        \hline
        $x^3 -(1+i)x^2 - 4i$ & $2^{16}\cdot5\cdot13$ & $10+10i$ \\
        \hline
        $x^3+(2+i)x^2+2$ & $2^{16}\cdot5\cdot17$ & $10$ \\
        \hline
        $x^3 -(3+i)x + (2+i)$ & $2^{16}\cdot97$ & $-2+i$ \\
        \hline
        $x^3+(4+2i)x^2-2x$ & $2^{18}\cdot 41$ & $4+8i$ \\
        \hline
        $x^3 + (-3+7i)x^2 + (1-i)x$ & $2^{19}\cdot5\cdot13$ & $-4+2i$ \\
        \hline
        $x^3 -(1-i)x^2-4i$ & $2^{16}\cdot17\cdot37$ & $8-16i$ \\
        \hline
    \end{tabular}
    \caption{Numerical verification of Deligne's Period Conjecture for motivic pieces of superelliptic curves $C/\QQ(i)\colon y^4 = f(x)$.}
    \label{tab:1}
\end{table}

In all of the examples presented in Table \ref{tab:1} we take $L_2(C^\tau,T)=1$. The second column displays the conductor of $\Res_{K/\QQ}C^\tau$, as in the functional equation \eqref{eqn:func2}. The final column displays a Gaussian integer whose difference from our computed value of $160 \cdot L(C^\tau,1)/\Omega_{C^\tau}$ has absolute value at most $10^{-10}$, which is less than $160^{-4}$---we ensure also that the output of Magma's \texttt{CFENew} is always less than this bound.

\begin{remark}
    The following observation allows us to obtain provably correct values for the conductor norms in Table \ref{tab:1}, assuming $L_2(C^\tau,T)=1$: the equation $D: \Delta^2 = \disc_x(y^4-f(x))$ defines a genus $3$ curve, the Jacobian of which decomposes under isogeny as a product of those of the genus $1$ curve $\Delta^2 = \disc_x(y^2-f(x))$ and the genus $2$ curve $H: \Delta^2 = y\cdot\disc_x(y^2-f(x))$. By \cite[Proposition 5.3]{WildConductors}, in combination with Lemma \ref{lem:conductors}, we deduce that the wild conductor exponent at $2$ of $C^\tau$ is half that of $H$. In turn, the latter can be computed using the results of \cite{DokDoris} alongside the associated implementation \cite{DokDorisRepo}.
\end{remark}

\begin{remark}
    Note that all of the polynomials appearing in Table \ref{tab:1} such that corresponding final column entry $a+bi$ satisfies $\gcd(2,a,b)=1$ have a rational root. This corresponds to a $4$-torsion point on $C/\QQ(i)\colon y^4=f(x)$. This seems to tentatively agree with the fact that the denominator in the BSD-quotient for $C$ is $|J_C(\QQ(i))_{\text{tors}}|^2$.
\end{remark}

\begin{remark}
    As before, the Galois equivariance part of Deligne's conjecture here would follow automatically from rationality. It would be of great interest to carry out analogous computations in the case of exponent five superelliptic curves, where there is a more interesting Galois action to verify on the $L$-values, but we are hindered by large conductor exponents at $5$ and so such examples seem to lie narrowly beyond the reach of the computational capacity available to the author.
\end{remark}

\bibliographystyle{plain}
%\nocite{*}
\bibliography{main} % Entries are in the refs.bib file

@article{Magma,
title = {{The Magma Algebra System I: The User Language}},
journal = {J. Symb. Comput.},
volume = {24},
number = {3},
pages = {235-265},
year = {1997},
author = {W. Bosma and J. Cannon and C. Playoust}
}

@misc{lmfdb,
  shorthand    = {LMFDB},
  author       = {The {LMFDB Collaboration}},
  title        = {The {$L$}-functions and modular forms database},
  howpublished = {\url{https://www.lmfdb.org}},
  year = {2026}
}

@article{DGKM2,
author = {V. Dokchitser and H. Green and A. Konstantinou and A. Morgan},
title = {{Parity of ranks of Jacobians of curves}},
year = {2025},
journal = {Proc. Lond. Math. Soc.},
volume = {131},
issue = {3}
}

@article{DokCalc,
  author={T. Dokchitser},
  title={Computing special values of motivic {$L$}-functions},
  journal={Experiment. Math.},
  volume={13},
  number={2},
  pages={137--149},
  year={2004}
}

@article{Deligne79,
  title={Valeurs de fonctions {$L$} et p{\'e}riodes d’int{\'e}grales},
  author={P. Deligne},
  year={1979},
  journal ={Proc. Symp. Pure Math.},
  volume = {33},
  pages = {313--346}
}

@article{DokchitserEvansWiersema,
url = {https://doi.org/10.1515/crelle-2020-0036},
title = {{On a BSD-type formula for $L$-values of Artin twists of elliptic curves}},
author = {V. Dokchitser and R. Evans and H. Wiersema},
pages = {199--230},
volume = {2021},
number = {773},
journal = {J. Reine Angew. Math.},
doi = {doi:10.1515/crelle-2020-0036},
year = {2021}
}

@article{BurnsMC,
  author          = {D. Burns and D. Macias Castillo},
  journal         = {Memoirs of Amer. Math. Soc.},
  title           = {{On refined conjectures of Birch and Swinnerton-Dyer type for Hasse--Weil--Artin $L$-series}},
  year={2024},
  volume={297},
  number={1482}
}

@article{EvansMCW,
  author = {R. Evans and D. Macias Castillo and H. Wiersema},
  title = {{Numerical Evidence for a Refinement of Deligne's Period Conjecture for Jacobians of Curves}},
  journal = {Experiment. Math.},
  volume = {34},
  number = {3},
  pages = {366--383},
  year = {2025}
}

@article{BurnsFlach,
  author          = {D. Burns and M. Flach},
  journal         = {Doc. Math.},
  volume          = {6},
  pages           = {501--570},
  title           = {{Tamagawa Numbers for Motives with (Non-Commutative) Coefficients}},
  year            = {2001}
}

@phdthesis{EvansThesis,
  author      = {R. Evans},
  school      = {King's College, London},
  title       = {{Artin-twists of Abelian Varieties: Periods, $L$-values and Arithmetic}},
  year        = {2021}
}

@article{KWSerreConj,
 author = {Khare, C. and Wintenberger, J.-P.},
 title = {Serre's modularity conjecture ({I})},
 journal = {Invent. Math.},
 volume = {178},
 number = {3},
 year = {2009}
}

@Inbook{Ribet2004,
author="Ribet, K. A.",
title="Abelian Varieties over $\mathbb{Q}$ and Modular Forms",
bookTitle="Modular Curves and Abelian Varieties",
year="2004",
publisher="Birkh{\"a}user",
address={Basel},
pages="241--261",
series = {Progr. Math.},
volume = {224}
}

@article{Ellenberg01,
title = {Endomorphism Algebras of {Jacobians}},
journal = {Adv. Math.},
volume = {162},
number = {2},
pages = {243--271},
year = {2001},
author = {J. S. Ellenberg}
}

@article{UlmerConductors,
    author = {D. Ulmer},
    title = {Conductors of $\ell$-adic representations},
    journal = {Proc. Am. Math. Soc.},
    volume={144},
    year = {2016}
}

@article{SuperellipticPeriods,
  journal = {Math. Comput.},
  title = {Computing period matrices and the {Abel--Jacobi} map of
superelliptic curves},
  author = {P. Molin and C. Neurohr},
  volume = {88},
  number = {316},
  year = {2019},
  pages = {847--888}
}

@misc{SuperellipticPeriodsCode,
  author = {P. Molin and C. Neurohr},
  howpublished = {\url{http://doi.org/10.5281/zenodo.1098275}},
  title = {{hcperiods: Arb and Magma packages for periods of
superelliptic curves}},
  year = {2017}
}

@article{Lupoian2Tors,
author = {E. Lupoian},
title = {{Two-torsion subgroups of some modular Jacobians}},
year = {2024},
volume = {20},
number = {10},
journal = {Int. J. Number Theory}
}

@misc{WildConductors,
author = {H. Spencer},
title = {Wild conductor exponents of curves},
year = {2025},
howpublished={\href{https://hspen99.github.io/writings}{Pre-print}}
}

@article{DokDoris,
author = {Dokchitser, T. and Doris, C.},
year = {2017},
title = {3-torsion and conductor of genus 2 curves},
volume = {88},
journal = {Math. Comput.}
}

@misc{DokDorisRepo,
author = {Doris, C.},
year = {2017},
title = {{Genus2Conductor} ({G}ithub repository)},
howpublished = {\url{https://github.com/cjdoris/Genus2Conductor}}
}

@misc{MaistretShukla2025,
author = {C. Maistret and H. Shukla},
title = {{On the factorization of twisted $L$-values and $11$-descents over $C_5$-number fields}},
year = {2025},
howpublished = {arXiv: \href{https://arxiv.org/abs/2501.09515}{{2501.09515}}}
}

@Inbook{Gross1982,
author="Gross, B. H.",
title="On the Conjecture of {Birch and Swinnerton-Dyer} for Elliptic Curves with Complex Multiplication",
bookTitle="Number Theory Related to Fermat's Last Theorem",
year="1982",
publisher="Birkh{\"a}user",
address="Boston, MA",
pages="219--236",
series="Progr. Math.",
volume="26"}

@article{KaniRosen,
  author = {E. Kani and M. Rosen},
  journal = {Math. Ann.},
  number = {2},
  title = {Idempotent relations and factors of {Jacobians}},
  volume = {284},
  year = {1989}
}

@article{MilneAV,
  author = {J. S. Milne},
  journal = {Invent. Math.},
  volume = {17},
  title = {On the {Arithmetic} of {Abelian Varieties}},
  year = {1972},
  pages = {177--190}
}

@misc{MotivicPiecesCode,
author = {H. Spencer},
year = {2025},
title = {{MotivicPiecesOfCurves} ({G}ithub repository)},
howpublished = {\url{https://github.com/hspen99/MotivicPiecesOfCurves}}
}

@article{DummitQuintics,
  author={D. S. Dummit},
  title={Solving solvable quintics},
  journal={Math. Comput.},
  volume={57},
  number={195},
  year = {1991},
  pages = {387--401}
}

@misc{Mathematica,
  author = {Wolfram Research{,} Inc.},
  title = {Mathematica, {V}ersion 14.3},
  url = {https://www.wolfram.com/mathematica},
  note = {Champaign, IL, 2025}
}

@misc{QuinticsCode,
  author = {D. S. Dummit},
  howpublished = {\url{http://site.uvm.edu/ddummit/files/2021/04/quintics.nb_.txt}}
}

@article{Roth,
author = {Roth, K. F.},
title = {Rational approximations to algebraic numbers},
journal = {Mathematika},
volume = {2},
number = {1},
pages = {1--20},
year = {1955}
}

@book{KhinchinContFractions,
  author    = {Khinchin, A. Ya.},
  title     = {Continued Fractions},
  publisher = {University of Chicago Press},
  address   = {Chicago},
  year      = {1964}
  }

\end{document}